\documentclass[12pt,fleqn] {article}
\usepackage{amssymb,amsmath,amsthm,setspace}
\usepackage{figsize}
\newcommand{\bs}{\boldsymbol}
\newcommand{\adb}{\allowdisplaybreaks}

\newcommand{\inv}{\frac{1}}

\newtheorem{thm}{Theorem}

\usepackage{multicol}
\usepackage{color}

\date{\today}

\title{Convergence rate to a lower tail dependence coefficient of a skew-$t$ distribution}
\author{Thomas Fung$^{a,}\footnote{Honorary Associate, University of Sydney.}$\,\, and Eugene Seneta$^{b,}$\footnote{Corresponding Author. Fax: + 61 2 93514534. Email address:  {\tt eseneta@maths.usyd.edu.au} (Eugene Seneta).} \\
{\small $^a$ Department of Statistics, Macquarie University, NSW 2109, Australia}\\
{\small $^b$ School of Mathematics and Statistics, University of Sydney, N.S.W. 2006, Australia}\\
}

\begin{document}
\maketitle

\begin{abstract}
We examine the rate of decay to the limit of the tail dependence coefficient of a bivariate skew $t$ distribution which always displays asymptotic tail dependence. It contains as a special case the usual bivariate symmetric $t$ distribution,  and hence is an appropriate (skew) extension. The rate is asymptotically power-law.  The second-order structure of the univariate quantile function for such a skew-$t$ distribution is a central  issue.

\end{abstract}


\section{Background and Motivation}
The coefficient of lower tail dependence of a random vector $\textbf{X} =
(X_1,X_2)^T$ with marginal inverse distribution function $F_1^{-1}$ and
$F_2^{-1}$ is defined as
\begin{equation}
\lambda_L = \lim_{u \rightarrow 0^+}\lambda_L(u), \quad \text{where} \quad \lambda_L(u) = P(X_1 \leq F_1^{-1}(u) | X_2
\leq F_2^{-1}(u)).\label{defn:tail dependence}
\end{equation}
$\textbf{X}$ is said to have asymptotic lower tail dependence if $\lambda_L$
exists positive. If $\lambda_L=0$, then $\textbf{X}$ is said to be
asymptotically independent in the lower tail. This quantity provides insight on
the tendency for the distribution to generate joint extreme event since it
measures the strength of dependence (or association) in the lower tails of a
bivariate distribution.  If the marginal distributions of these random variables are
continuous, then from (\ref{defn:tail dependence}), it follows that
$\lambda_L(u)$ can be expressed in terms of  the copula of
$\textbf{X}$, $C(u_1,u_2)$, as
\begin{equation}
\lambda_L(u) = \frac{P(X_1\leq F_1^{-1}(u), X_2\leq F_2^{-1}(u))}{P(X_2 \leq F_2^{-1}(u))} =  \frac{C(u,u)}{u}.
\label{defn:tail dependence 3}
\end{equation}


In this paper we  investigate  the rate of convergence to 0 of $\lambda_L(u)-\lambda_L$ as $u\to 0^+$,  in an important case 
when $\lambda_L >0.$ Heffernan (2000) provides a summary of coefficients for many commonly employed bivariate distributions, but the specific situation which  we study is not considered.

In the sequel we refer to the  bivariate skew-$t$  as that distribution resulting  from variance-mixing of the bivariate skew
normal, $\textbf{Z}\sim SN_2(\bs{\theta},R)$ (see Azzalini and Dalla Valle
(1996)), inversely with a gamma random variable $V \sim
\Gamma(\frac{\eta}{2},\frac{\eta}{2})$, with $\eta>0$:
\begin{equation}
\textbf{X} = V^{-\inv{2}}\textbf{Z}, \label{defn:azzalini skew t}
\end{equation}
where $\textbf{Z}$ is independently distributed of $V$.
 
This skew distribution was originally introduced in multivariate form in Branco and Dey (2001) and studied extensively in Azzalini and Capitanio (2003). Some recent reviews on this area of study can be found in Azzalini and Genton (2008), Azzalini and Capitanio (2010) and in the book edited by Genton (2004). 

The bivariate skew-$t$ always satisfies $\lambda_L >0$ (See Fung and Seneta (2010)). This was also considered in Bortot (2010) and Padoan (2011) with an approach initiated by Cheng and Genton (2007) which is quite different from that of Fung and Seneta (2010). The  case $\theta_1 = \theta_2 = 0$ reduces to the symmetric bivariate $t$ distribution. In this sense, the bivariate skew-$t$ distribution defined by (\ref{defn:azzalini skew t}) is a  more  appropriate generalisation of the symmetric case.

The motivation for our investigation of the rate of convergence in the present specific case of bivariate skew-$t$ arose from the following. Ramos and Ledford (2009), continuing the work of Ledford and Tawn (1997), studied intensively a family of bivariate distributions (which they characterised) which satisfied in particular the condition
\begin{equation}
\lambda_L(u) = u^{\inv{\alpha}-1} L(u). \label{RL condition}
\end{equation}
Here $L(u)$ is a slowly varying function as $u\to 0^+$, and $\alpha\in (0,1]$, so that, in fact, the value of $\alpha$ could be used for comparison of the degree of tail dependence structure between members of the family. The standard bivariate extreme value models correspond to $\alpha=1$.

Expression (\ref{RL condition}) may also be regarded as the rate of convergence to $\lambda_L(u)$ when $ \lambda_L =0$, but when $\lambda_L(u) \to \lambda_L> 0$, which is also covered by (\ref{RL condition})   with $\alpha =1$ and $L(u) \to \lambda_L$, the rate of convergence  is more appropriately studied by considering the rate of convergence to $0$ as $u\to 0+$ of  $$| \lambda_L(u) - \lambda_L|.$$  Our study of an important special case is an early step in this direction.

\section{The Bivariate Skew-{\it t} Distribution}

From Branco and Dey (2001), the random vector $\textbf{X}$, defined by (\ref{defn:azzalini skew t}), has
probability density:
\begin{equation*}
f_{\textbf{X}}(\textbf{x}) =
\frac{2\Gamma(\frac{\eta+2}{2}) \left(1+\frac{\textbf{x}^TR^{-1}\textbf{x}}{\eta}\right)^{-\frac{\eta+2}{2}}}{\pi\eta\Gamma(\frac{\eta}{2})\sqrt{1-\rho^2}}F_{t_{\eta+2}}\left(\bs{\theta}^T\textbf{x}\sqrt{\frac{\eta+2}{\eta+\textbf{x}^TR^{-1}\textbf{x}}}\right),
\end{equation*}
where $F_{t_{\eta+2}}(\cdot)$ is the distribution function of the (symmetric)
$t$ distribution with $\eta+2$ degrees of freedom, $R = \left(\begin{smallmatrix} 1 & & \rho \\ & & \\\rho & &
1\end{smallmatrix} \right)$, and $\bs{\theta} = (\theta_1, \theta_2)^T$ is a vector
that controls the asymmetry of the distribution.

%


The marginal density of $X_1$ can then be found as
\begin{equation*}
f_{X_1}(x) = 2f_{t_{\eta}}(x)F_{t_{\eta+1}}(\lambda_1x\sqrt{\frac{\eta+1}{\eta+x^2}}),
\end{equation*}
where 
\begin{equation}
f_{t_{\eta}}(x) = \frac{\Gamma(\frac{\eta+1}{2})}{(\pi\eta)^{\inv{2}}\Gamma(\frac{\eta}{2})}\left(1+\frac{x^2}{\eta}\right)^{-\frac{\eta+1}{2}} \label{symmetric t density}
\end{equation}
is the density of the (symmetric) $t$ distribution with $\eta$ degrees of freedom and 
$\lambda_1 = (\theta_1+\rho\theta_2)/\sqrt{1+\theta_2^2(1-\rho^2)}$. 
$X_2$ has a similar marginal density, except its marginal skewness
parameter, $\lambda_2,$ takes the form of $\lambda_2 = (\theta_2+\rho\theta_1)/\sqrt{1+\theta_1^2(1-\rho^2)}.$

From (\ref{defn:tail dependence 3}) and using some basic properties of copulas, it can be shown that 
\begin{align}
\notag &\lambda_L = \lim_{u \rightarrow 0^+}\lambda_L(u) \\
 = &\lim_{y\rightarrow -\infty}\bigl[ P(X_2\leq F_2^{-1}(F_1(y))|X_1= y) +P(X_1\leq F_1^{-1}(F_2(y))|X_2 = y)\bigr]. \label{temp1}
\end{align}
Fung and Seneta (2010) showed that if $\textbf{X} = (X_1, X_2)^T$ is a random vector defined by (\ref{defn:azzalini skew t}), then \begin{align}
\notag  &\lim_{y\rightarrow -\infty} P(X_2\leq F_2^{-1}(F_1(y))|X_1= y)\\
=& \int^{-a_{2.1}}_{-\infty} f_{t_{\eta+1}}(z) \frac{F_{t_{\eta+2}}
\label{first half:azzalini skew t} \left(
\Bigl(\theta_2\sqrt{\frac{(1-\rho^2)}{\eta+1}} z -(\theta_1+\rho\theta_2)
\Bigr)
\sqrt{\frac{\eta+2}{1+\frac{z^2}{\eta+1}}}\right)}{F_{t_{\eta+1}}(-\lambda_1
\sqrt{\eta+1})}\,dz;\\
\notag \text{and } \quad & \lim_{y\rightarrow -\infty} P(X_1\leq F_1^{-1}(F_2(y))|X_2 = y) \\
=& \int^{-a_{1.2}}_{-\infty} \frac{f_{t_{\eta+1}}(z) F_{t_{\eta+2}} \left( \Bigl(\theta_1
\sqrt{\frac{(1-\rho^2)}{\eta+1}} z -(\theta_2+\rho\theta_1) \Bigr)
\sqrt{\frac{\eta+2}{1+\frac{z^2}{\eta+1}}}\right)}{F_{t_{\eta+1}}(-\lambda_2
\sqrt{\eta+1})}\,dz; \label{second half:azzalini skew t}
\end{align}
where 
$a_{2.1}=\bigl((\frac{F_{t_{\eta+1}}(-\lambda_2\sqrt{\eta+1})}{F_{t_{\eta+1}}(-\lambda_1\sqrt{\eta+1})})^{\inv{\eta}}-\rho\bigr)\sqrt{\frac{\eta+1}{1-\rho^2}}$, and  $a_{1.2}
=\bigl((\frac{F_{t_{\eta+1}}(-\lambda_1\sqrt{\eta+1})}{F_{t_{\eta+1}}(-\lambda_2\sqrt{\eta+1})})^{\inv{\eta}}-\rho\bigr)\sqrt{\frac{\eta+1}{1-\rho^2}}.$


We shall  show that
\begin{equation}
|\lambda_L(u) - \lambda_L| = \left| \frac{C(u,u)}{u} - \lambda_L\right| \sim {\mbox{Const.}} u^{\frac{2}{\eta}}, \label{general approach}
\end{equation}
 as $u\to 0^+$.

The rest of this paper is set out as follows. In Section 3, we derive an accurate lower quantile result for the skew $t$ distribution defined in (\ref{defn:azzalini skew t}).
In Section 4, we derive the rate of convergence in the form of (\ref{general approach}) for the skew-$t$ distribution.

\section{Lower Quantile results}
In our subsequent theoretical development, both the asymptotic behaviour of $F_i(y)$ and its inverse $F_i^{-1}(y)$ as $y\to -\infty$ with higher order terms are needed. We begin by discussing the behaviour of $F_i(y)$ as $y\to -\infty$. Without loss of generality, set $i = 1.$ The result is summarised into the following theorem, only the first term of which  is given  in eqn.(28)  of Fung and Seneta (2010).

\begin{thm} The asymptotic behaviour of the marginal distribution function of $X_1$ is
\begin{align}
F_1(y) &=P(X_1\leq y) = c_1|y|^{-\eta}(1+d_1y^{-2}+O(y^{-4})),\quad   \text{as $y\to -\infty$}, \label{asym F_1}\\
\notag \text{where}\quad c_1 &= \frac{2\Gamma(\frac{\eta+1}{2})\eta^{\frac{\eta+1}{2}}}{(\pi\eta)^{\inv{2}}\Gamma(\frac{\eta}{2})}\frac{F_{t_{\eta+1}}(-\lambda_1\sqrt{\eta+1})}{\eta},\\
\notag d_1 &=  - \frac{\eta^2(\eta+1)}{2(\eta+2)}+ \frac{\eta^2 f_{t_{\eta+1}}(-\lambda_1\sqrt{\eta+1})\lambda_1\sqrt{\eta+1}}{2(\eta+2)F_{t_{\eta+1}}(-\lambda_1\sqrt{\eta+1})}.
\end{align}
\end{thm}
\textbf{\emph{Proof.}}
For $x<0$ and by using a second-order Mean Value Theorem,
we have {\adb
\begin{align}
\notag & F_{t_{\eta+1}}(\lambda_1x\sqrt{\frac{\eta+1}{\eta+x^2}}) = F_{t_{\eta+1}}(-\lambda_1\sqrt{\eta+1}(1+\frac{\eta}{x^2})^{-\inv{2}})\\
\notag =& F_{t_{\eta+1}}(-\lambda_1\sqrt{\eta+1}) + f_{t_{\eta+1}}(-\lambda_1\sqrt{\eta+1})\lambda_1\sqrt{\eta+1}[1-(1+\frac{\eta}{x^2})^{-\inv{2}}]\\
\notag &\quad + (\lambda_1\sqrt{\eta+1})^2[1-(1+\frac{\eta}{x^2})^{-\inv{2}}]^2F''_{t_{\eta+1}}(\delta_1(x))/2, 
\end{align}}
for some $\delta_1(x)$ contained in the interval $(\min(-\lambda_1\sqrt{\eta+1}(1+\frac{\eta}{x^2})^{-\inv{2}}, -\lambda_1\sqrt{\eta+1})$, $\max(-\lambda_1\sqrt{\eta+1}(1+\frac{\eta}{x^2})^{-\inv{2}}, -\lambda_1\sqrt{\eta+1}))$;
\begin{align*}
&= F_{t_{\eta+1}}(-\lambda_1\sqrt{\eta+1}) + f_{t_{\eta+1}}(-\lambda_1\sqrt{\eta+1})\lambda_1\sqrt{\eta+1}[\frac{\eta}{2x^2} + O(\inv{x^4})]\\
& \quad +(\lambda_1\sqrt{\eta+1})^2[1-(1+\frac{\eta}{x^2})^{-\inv{2}}]^2F''_{t_{\eta+1}}(\delta_1(x))/2.\\
\text{Since } & |F''_{t_{\eta+1}}(\delta_1(x))|  = \left| \frac{\Gamma(\frac{\eta+1}{2})}{(\pi\eta)^{\inv{2}}\Gamma(\frac{\eta}{2})}(\frac{\eta+1}{\eta})\delta_1(x)(1+\frac{\delta_1^2(x)}{\eta})^{-(\frac{\eta+3}{2})}\right| \leq k_1
\end{align*}
for some constant $k_1$ as the function is bounded for large $|x|$ and
\begin{equation*}
(1-(1+\frac{\eta}{x^2})^{-\inv{2}})^2 = O(\inv{x^4}).
\end{equation*}
Therefore the dominating term of
\begin{equation*}
(\lambda_1\sqrt{\eta+1})^2[1-(1+\frac{\eta}{x^2})^{-\inv{2}}]^2F''_{t_{\eta+1}}(\delta_1(x))/2
\end{equation*}
is in the order of $x^{-4}$ and hence,
\begin{align*}
& F_{t_{\eta+1}}(\lambda_1x\sqrt{\frac{\eta+1}{\eta+x^2}}) \\
=& F_{t_{\eta+1}}(-\lambda_1\sqrt{\eta+1}) + f_{t_{\eta+1}}(-\lambda_1\sqrt{\eta+1})\lambda_1\sqrt{\eta+1}\frac{\eta}{2x^2}\left(1 + O(\inv{x^2})\right).
\end{align*}
Then for any $y< 0$, 
\begin{align}
\notag & F_1(y) = P(X_1\leq y) = \int^{y}_{-\infty}2 f_{t_{\eta}}(x)F_{t_{\eta+1}}(\lambda_1 x\sqrt{\frac{\eta+1}{\eta+x^2}})\,dx\\
=& \int^{y}_{-\infty}2 f_{t_{\eta}}(x)F_{t_{\eta+1}}(-\lambda_1\sqrt{\eta+1})\,dx \label{two terms:1}\\
& + \int^{y}_{-\infty}f_{\eta}(x) f_{t_{\eta+1}}(-\lambda_1\sqrt{\eta+1})\lambda_1\sqrt{\eta+1}x^{-2}\left(1+O(\inv{x^2})\right)\,dx \label{two terms:2}
\end{align}
We shall consider these two terms separately. Focusing on the first term, i.e. 

\noindent (\ref{two terms:1}), we have
\begin{equation*}
\int^{y}_{-\infty}2 f_{t_{\eta}}(x)F_{t_{\eta+1}}(-\lambda_1\sqrt{\eta+1})\,dx=\int^{\infty}_{|y|}c\, x^{-(\eta+1)}(1+\frac{\eta}{x^2})^{-(\frac{\eta+1}{2})}\,dx
\end{equation*}
by setting $c= \frac{2\Gamma(\frac{\eta+1}{2})\eta^{\frac{\eta+1}{2}}}{(\pi\eta)^{\inv{2}}\Gamma(\frac{\eta}{2})}F_{t_{\eta+1}}(-\lambda_1\sqrt{\eta+1})$;
\begin{equation*}
=c\frac{|y|^{-\eta}}{\eta}\Bigl(1-(\frac{\eta+1}{2})\frac{\eta}{y^2} + O(\inv{y^4})\Bigr)+ c(\eta+1)\Bigl\{ \frac{|y|^{-(\eta+2)}}{\eta+2}\bigl(1+O(\inv{y^2})\bigr)\Bigr\}
\end{equation*}
by applying integration by parts as suggested in Soms (1976).
Thus, {\adb
\begin{align}
\notag  & \int^{y}_{-\infty}2f_{t_{\eta}}(x)F_{t_{\eta+1}}(-\lambda_1\sqrt{\eta+1}) \,dx\\
& = \frac{2\Gamma(\frac{\eta+1}{2})\eta^{\frac{\eta+1}{2}}}{(\pi\eta)^{\inv{2}}\Gamma(\frac{\eta}{2})}F_{t_{\eta+1}}(-\lambda_1\sqrt{\eta+1})\frac{|y|^{-\eta}}{\eta}\bigl(1+ \frac{\frac{\eta(\eta+1)}{\eta+2}- \frac{\eta(\eta+1)}{2}}{y^2}+ O(\inv{y^4})\bigr). \label{two terms:1 end}
\end{align}}
The second term, i.e. (\ref{two terms:2}), can be treated similarly to get
\begin{align}
\notag &  \int^{y}_{-\infty}f_{t_{\eta}}(x) f_{t_{\eta+1}}(-\lambda_1\sqrt{\eta+1})\lambda_1\sqrt{\eta+1}x^{-2}\left(1+O(\inv{x^2})\right)\,dx\\
=& \frac{\Gamma(\frac{\eta+1}{2})\eta^{\frac{\eta+3}{2}}}{(\pi\eta)^{\inv{2}}\Gamma(\frac{\eta}{2})}f_{t_{\eta+1}}(-\lambda_1\sqrt{\eta+1})\lambda_1\sqrt{\eta+1}\frac{|y|^{-(\eta+2)}}{\eta+2}\left(1+O(\inv{y^2})\right), \label{two terms:2 end}
\end{align}
Hence, by combining (\ref{two terms:1 end}) and (\ref{two terms:2 end}) the result follows. \hfill $\Box$


\begin{thm}The inverse of $P(X_1 \leq y)$, $F_1^{-1}(u)$, satisfies:
\begin{equation}
F_1^{-1}(u)  = - c_1^{\inv{\eta}}u^{-\inv{\eta}}(1+\frac{d_1}{c_1^{\frac{2}{\eta}}}\frac{u^{\frac{2}{\eta}}}{\eta}+O(u^{\frac{4}{\eta}})), \quad \text{as $u\to 0^+$}. \label{asym F_1^-1}
\end{equation}
\end{thm}
\begin{proof}
On account of (\ref{asym F_1}), to find the inverse of $F_1(\cdot)$ i.e. $F_1^{-1}(\cdot)$, it is sufficient to consider the function $H(y)=F_1(-y),\,  y >0,$ so that, by Theorem 1
\begin{equation*}
H(y) =  c_1y^{-\eta}(1+d_1y^{-2}+O(y^{-4})).
\end{equation*}
where $\eta>0$ and $c_1$, $d_1\ne 0$, as $y\to \infty$. Now define $G(y) = 1/H(y)$  so that 
\begin{equation}
G(y) = c_1^{-1}y^{\eta}(1-d_1y^{-2}+O(y^{-4})) =c_1^{-1}y^{\eta}S(y)  \label{svf1}
\end{equation}which defines $S(y)$, and  we note $S(y) \to 1,$ as   $y \to \infty.$
Noting that $G(y)$ is strictly increasing and continuous, denote its inverse by $G^{\leftarrow}(y)$. (We shall use this notation for inverses, to avoid confusion, only in  this proof.)  Then:
\begin{equation*}
 y = G(G^{\leftarrow}(y)) = c^{-1}_1(G^{\leftarrow}(y))^{\eta}\Bigl(1-d_1(G^{\leftarrow}(y))^{-2}+O\bigl((G^{\leftarrow}(y))^{-4}\bigr)\Bigr) \end{equation*}
so that  {\adb
\begin{eqnarray}
G^{\leftarrow}(y)&=&\Bigg{\{} \frac{c_1y}{\Bigl(1-d_1(G^{\leftarrow}(y))^{-2}+O\bigl((G^{\leftarrow}(y))^{-4}\bigr)\Bigr)}\Bigg{\}}^{\frac{1}{\eta}} \nonumber \\
 &=& c_1^{\inv{\eta}}y^{\inv{\eta}}\Bigl(1+\frac{d_1}{\eta}(G^{\leftarrow}(y))^{-2}+O\bigl((G^{\leftarrow}(y))^{-4}\Bigr)\Bigr) \label{svf3}\\
 &=& (c_1y)^{\frac{1}{\eta}}S^{*}(y), \label{svf2}
\end{eqnarray}}
which defines $S^*(y)$. 
Then using (\ref{svf2}) and (\ref{svf1})
\begin{equation*}
 y = G^{\leftarrow}(G(y)) =((G(y)c_1)^{\inv{\eta}}S^{*}(G(y)) =(c_1^{-1}y^{\eta}S(y)c_1)^{\inv{\eta}}S^{*}(G(y)) \end{equation*} so that  $S^{*}(G(y))= S^{- \inv{\eta}}(y)$, whence, since $S(y) \to 1$ as $y \to \infty$ 
\begin{equation} 
 \lim_{y \to \infty} S^{*}(y) =1.  \label{svf6}
\end{equation}
Hence, substituting expression (\ref{svf3}) for $G^{\leftarrow}(y)$ into the right hand side of  (\ref{svf3}) (recursively), and using (\ref{svf6}), as $y \to \infty$, 
\begin{equation*}
 G^{\leftarrow}(y) = c_1^{\inv{\eta}}y^{\inv{\eta}}\Bigl(1+\frac{d_1}{c_1^{\frac{2}{\eta}}\eta}y^{-\frac{2}{\eta}}+O(y^{-\frac{4}{\eta}})\Bigr).
 \end{equation*} The final result follows as $H(y) = 1/G(y)$ implies that $H^{\leftarrow}(y) = G^{\leftarrow}(1/y)$. \end{proof}
 The representations (\ref{svf1}) and (\ref{svf2}) are those for a regularly varying function with index $\eta$,  and  its inverse $G^{\leftarrow}(\cdot)$, regularly  varying with index $1/\eta$. (See Proposition 0.8 on p. 22 of Resnick (1987)).
However, the specialized form of the slowly varying  function $S(y)$ needs to be invoked in our self-contained proof.

Similarly, the inverse of $P(X_2\leq y)$, i.e. $F_2^{-1}(u),$ is thus
\begin{equation}
F_2^{-1}(u) = - c_2^{\inv{\eta}}u^{-\inv{\eta}}(1+\frac{d_2}{c_2^{\frac{2}{\eta}}}\frac{u^{\frac{2}{\eta}}}{\eta}+O(u^{\frac{4}{\eta}})), \quad \text{as $u\to 0^+$},\label{asym F_2^-1}
\end{equation}
where {\adb
\begin{align*}
c_2 =& \frac{2\Gamma(\frac{\eta+1}{2})\eta^{\frac{\eta+1}{2}}}{(\pi\eta)^{\inv{2}}\Gamma(\frac{\eta}{2})}\frac{F_{t_{\eta+1}}(-\lambda_2\sqrt{\eta+1})}{\eta}\\
\text{and }\quad d_2 =&  - \frac{\eta^2(\eta+1)}{2(\eta+2)}+ \frac{\eta^2 f_{t_{\eta+1}}(-\lambda_2\sqrt{\eta+1})\lambda_2\sqrt{\eta+1}}{2(\eta+2)F_{t_{\eta+1}}(-\lambda_2\sqrt{\eta+1})}.
\end{align*}}
A result which we  shall need  repeatedly in the sequel is that
\begin{equation}
c(y)\stackrel{def}{=} F_2^{-1}(F_1(y)) = \left(\frac{F_{t_{\eta+1}}(-\lambda_2\sqrt{\eta+1})}{F_{t_{\eta+1}}(-\lambda_1\sqrt{\eta+1})}\right)^{\inv{\eta}} y\bigl(1-\frac{d_1-d_2(\frac{c_1}{c_2})^{\frac{2}{\eta}}}{\eta\,y^2}+ O(y^{-4})\bigr)\label{asym F_2F_1}
\end{equation}
as $y \to -\infty ,$ which follows after some algebra by combining (\ref{asym F_1}) and (\ref{asym F_2^-1}).

Notice that when $\lambda_1 = \lambda_2 =\lambda$, then the first order  term in (\ref{asym F_2F_1}) vanishes as {\adb
\begin{align*}
d_1 - d_2(\frac{c_1}{c_2})^{\frac{2}{\eta}}=&  - \frac{\eta^2(\eta+1)}{2(\eta+2)}+ \frac{\eta^2 f_{t_{\eta+1}}(-\lambda_1\sqrt{\eta+1})\lambda_1\sqrt{\eta+1}}{2(\eta+2)F_{t_{\eta+1}}(-\lambda_1\sqrt{\eta+1})}-(\frac{F_{t_{\eta+1}}(-\lambda_1\sqrt{\eta+1})}{F_{t_{\eta+1}}(-\lambda_2\sqrt{\eta+1})})^{\frac{2}{\eta}}\\
& \quad \times (- \frac{\eta^2(\eta+1)}{2(\eta+2)}+ \frac{\eta^2 f_{t_{\eta+1}}(-\lambda_2\sqrt{\eta+1})\lambda_2\sqrt{\eta+1}}{2(\eta+2)F_{t_{\eta+1}}(-\lambda_2\sqrt{\eta+1})})\\
=&
 - \frac{\eta^2(\eta+1)}{2(\eta+2)}+ \frac{\eta^2 f_{t_{\eta+1}}(-\lambda\sqrt{\eta+1})\lambda\sqrt{\eta+1}}{2(\eta+2)F_{t_{\eta+1}}(-\lambda\sqrt{\eta+1})}\\
& \quad - (- \frac{\eta^2(\eta+1)}{2(\eta+2)}+ \frac{\eta^2 f_{t_{\eta+1}}(-\lambda\sqrt{\eta+1})\lambda\sqrt{\eta+1}}{2(\eta+2)F_{t_{\eta+1}}(-\lambda\sqrt{\eta+1})}) = 0,
\end{align*}}
so that  $\lambda_1 = \lambda_2  \Rightarrow c(y) = F_2^{-1}(F_1(y)) = y(1+O(y^{-3})).$
Finally, one can show that:
$
\lambda_1 = \lambda_2  \Leftrightarrow \theta_1=\theta_2.$
This case of ``equiskewness" in particular covers the  symmetric case $\theta_1 = \theta_2 = 0.$

\section{Main result}

\begin{thm} For the bivariate  skew-t distribution:
\begin{equation*}
|\lambda_L(u) - \lambda_L |  = u^{\frac{2}{\eta}}L(u), 
\end{equation*}
where $L(u) \to k,$ where $k$ is a constant  as $u\to 0^+$.  
\end{thm}
\begin{proof}
From (\ref{defn:tail dependence 3}) and using some basic properties of copulas, we have {\adb
\begin{align}
\notag & \frac{dC(u,u)}{d\,u} - \lambda_L \\
&= \{P(X_2\leq F_2^{-1}(u)|X_1 = F_1^{-1}(u))  - \lim_{u\to 0^+} P(X_2\leq F_2^{-1}(u)|X_1 = F_1^{-1}(u))\} \label{origianl expression 1}\\
& \quad +\{P(X_1\leq F_1^{-1}(u)|X_2 = F_2^{-1}(u))  - \lim_{u\to 0^+} P(X_1\leq F_1^{-1}(u)|X_2 = F_2^{-1}(u))\} \label{original expression 2}
\end{align}}
which allows for the distribution being skew. 
Without loss of generality, we focus on (\ref{origianl expression 1}). Applying a change of variable of $y =  F_1^{-1}(u),$ so that $y\to -\infty$ as $u\to 0^+$, (\ref{origianl expression 1}) becomes
\begin{equation*}
P(X_2\leq c(y)|X_1 =y)  - \lim_{y\to -\infty} P(X_2\leq c(y)|X_1 = y).
\end{equation*}
Once again from Fung and Seneta (2010), these  two terms can be expressed respectively as
{\adb
\begin{align}
\notag & P(X_2\leq c(y)|X_1 = y) = \int^{L_1(y)}_{-\infty} f_{t_{\eta+1}}(z) \tau(z,y)\,dz, \\
=& \int^{L_1(y)}_{L_1}f_{t_{\eta+1}}(z)\tau(z,y)\,dz + \int^{L_1}_{-\infty}f_{t_{\eta+1}}(z)\tau(z,y)\,dz,
\label{before first half:azzalini skew t}
\end{align}}
and {\adb
\begin{align}
\notag   \lim_{y\rightarrow -\infty} P(X_2\leq c(y)|X_1= y) =& \int^{L_1}_{-\infty} f_{t_{\eta+1}}(z)\left\{\frac{F_{t_{\eta+2}}
\left( a(z)+b(z)\right)}{F_{t_{\eta+1}}(-\lambda_1
\sqrt{\eta+1})}\right\}\,dz \\
\notag =& \int^{L_1}_{-\infty}f_{t_{\eta+1}}(z)\tau(z,*)\,dz 
\end{align}}
where
\begin{align*}
& L_1(y) = \frac{c(y)-\rho y}{\left(\frac{(\eta+y^2)(1-\rho^2)}{\eta+1}\right)^{\inv{2}}},  \quad c(y) = F_2^{-1}(F_1(y)),\quad a(z) = \theta_2\sqrt{\frac{(1-\rho^2)}{\eta+1}}\sqrt{\frac{\eta+2}{1+\frac{z^2}{\eta+1}}}z,\\
& b(z) = -(\theta_1+\rho\theta_2) \sqrt{\frac{\eta+2}{1+\frac{z^2}{\eta+1}}}, \quad \tau(z,y) = \frac{F_{t_{\eta+2}}(a(z)+b(z)(1+\frac{\eta}{y^2})^{-\inv{2}})}{F_{t_{\eta+1}}(-\lambda_1\sqrt{\eta+1}(1+\frac{\eta}{y^2})^{-\inv{2}})},\\
& \tau(z,*) = \frac{F_{t_{\eta+2}}
\left( a(z)+b(z)\right)}{F_{t_{\eta+1}}(-\lambda_1\sqrt{\eta+1})} = \lim_{y\to -\infty}\tau(z,y), \text{ and $f_{t_{\eta+1}}(z)$ is defined by (\ref{symmetric t density}).}
\end{align*}
Lastly,
\begin{align*}
L_1= \lim_{y\to -\infty}L_1(y) = -\left\{\left(\frac{F_{t_{\eta+1}}(-\lambda_2\sqrt{\eta+1})}{F_{t_{\eta+1}}(-\lambda_1\sqrt{\eta+1})}\right)^{\inv{\eta}}-\rho\right\}\sqrt{\frac{\eta+1}{1-\rho^2}},
\end{align*}
by using (\ref{asym F_2F_1}). Notice that we made no assumption that $L_1(y) > L_1$ and the integral in (\ref{before first half:azzalini skew t}) is still valid if $L_1(y) \leq L_1$ as $\int^{L_1(y)}_{L_1}$ is equivalent to $-\int^{L_1}_{L_1(y)}$. Thus, 
{\adb
\begin{align}
\notag & P(X_2 \leq F_2^{-1}(F_1(y))| X_1= y) - \lim_{y\to -\infty}P(X_2 \leq F_2^{-1}(F_1(y))| X_1= y) \\
\notag =& \int^{L_1(y)}_{-\infty} f_{t_{\eta+1}}(z)\tau(z,y)\,dz - \int^{L_1}_{-\infty}f_{t_{\eta+1}}(z)\tau(z,*)\,dz \\
 =& \int^{L_1}_{-\infty}f_{t_{\eta+1}}(z)\{\tau(z,y)-\tau(z,*)\}\,dz + \int^{L_1(y)}_{L_1}f_{t_{\eta+1}}(z)\tau(z,y)\,dz. \label{II}
\end{align}}
Treating  these two summands separately, after some  algebra,{\adb
\begin{align}
\notag & \int^{L_1}_{-\infty}f_{t_{\eta+1}}(z)\{\tau(z,y)-\tau(z,*)\}\,dz \\
\notag \sim&  -\frac{\eta}{2}y^{-2}\int^{L_1}_{-\infty}f_{t_{\eta+1}}(z)\Biggl\{ \frac{f_{t_{\eta+1}}(-\lambda_1\sqrt{\eta+1})F_{t_{\eta+2}}(a(z)+b(z))}{F_{t_{\eta+1}}(-\lambda_1\sqrt{\eta+1})}\lambda_1\sqrt{\eta+1}\\
& \quad + f_{t_{\eta+2}}(a(z)+b(z))b(z)\Biggr\}/F_{t_{\eta+1}}(-\lambda_1\sqrt{\eta+1})\,dz, \quad \text{as $y\to -\infty$.} \label{I final}
\end{align}}
Next, considering the second term of (\ref{II}), by the mean value theorem,
\begin{align}
\notag & \int^{L_1(y)}_{L_1} f_{t_{\eta+1}}(z)\tau(z,y)\,dz = (L_1(y) - L_1)f_{t_{\eta+1}}(\xi_{y})\tau(\xi_{y},y) \\
\notag =& \frac{y^{-2}}{(\frac{1-\rho^2}{\eta+1})^{\inv{2}}}\left\{ \left(\frac{F_{t_{\eta+1}}(-\lambda_2\sqrt{\eta+1})}{F_{t_{\eta+1}}(-\lambda_1\sqrt{\eta+1})}\right)^{\inv{\eta}}\left(\frac{d_1-d_2(\frac{c_1}{c_2})^{\frac{2}{\eta}}}{\eta}\right)+ \frac{\eta}{2}(\left(\frac{F_{t_{\eta+1}}(-\lambda_2\sqrt{\eta+1})}{F_{t_{\eta+1}}(-\lambda_1\sqrt{\eta+1})}\right)^{\inv{\eta}}-\rho)\right\}\\
\notag & \quad \times f_{t_{\eta+1}}(\xi_{y})\tau(\xi_{y},y) (1+O(\inv{y}))\\
\notag \sim& \frac{y^{-2}}{(\frac{1-\rho^2}{\eta+1})^{\inv{2}}}\left\{ \left(\frac{F_{t_{\eta+1}}(-\lambda_2\sqrt{\eta+1})}{F_{t_{\eta+1}}(-\lambda_1\sqrt{\eta+1})}\right)^{\inv{\eta}}\left(\frac{d_1-d_2(\frac{c_1}{c_2})^{\frac{2}{\eta}}}{\eta}\right)- \frac{\eta}{2}(\left(\frac{F_{t_{\eta+1}}(-\lambda_2\sqrt{\eta+1})}{F_{t_{\eta+1}}(-\lambda_1\sqrt{\eta+1})}\right)^{\inv{\eta}}-\rho)\right\}\\
& \quad \times f_{t_{\eta+1}}(L_1)\frac{F_{t_{\eta+2}}(a(L_1)+b(L_1))}{F_{t_{\eta+1}}(-\lambda_1\sqrt{\eta+1})}(1+O(\inv{y})), \label{II final}
\end{align}
as $y\to -\infty$. Subsequently, if we combine (\ref{I final}) and (\ref{II final}), we have {\adb
\begin{align*}
&P(X_2\leq c(y)| X_1 = y) - \lim_{y\rightarrow -\infty} P(X_2 \leq c(y) |X_1  =y) \\
\sim&
y^{-2}\Biggl\{ -\int^{L_1}_{-\infty}f_{t_{\eta+1}}(z)\biggl\{ \frac{f_{t_{\eta+1}}(-\lambda_1\sqrt{\eta+1})F_{t_{\eta+2}}(a(z)+b(z))}{F_{t_{\eta+1}}(-\lambda_1\sqrt{\eta+1})}\lambda_1\sqrt{\eta+1}\frac{\eta}{2}\\
& \quad + f_{t_{\eta+2}}(a(z)+b(z))b(z)\frac{\eta}{2}\biggr\}/F_{t_{\eta+1}}(-\lambda_1\sqrt{\eta+1})\,dz + \frac{1}{(\frac{1-\rho^2}{\eta+1})^{\inv{2}}}\\
& \quad \times \biggl\{ \left(\frac{F_{t_{\eta+1}}(-\lambda_2\sqrt{\eta+1})}{F_{t_{\eta+1}}(-\lambda_1\sqrt{\eta+1})}\right)^{\inv{\eta}}\left(\frac{d_1-d_2(\frac{c_1}{c_2})^{\frac{2}{\eta}}}{\eta}\right)+ \frac{\eta}{2}(\left(\frac{F_{t_{\eta+1}}(-\lambda_2\sqrt{\eta+1})}{F_{t_{\eta+1}}(-\lambda_1\sqrt{\eta+1})}\right)^{\inv{\eta}}-\rho)\biggr\}\\
& \quad \times f_{t_{\eta+1}}(L_1)\frac{F_{t_{\eta+2}}(a(L_1)+b(L_1))}{F_{t_{\eta+1}}(-\lambda_1\sqrt{\eta+1})}\Biggr\} = k_{2.1} y^{-2}.
\end{align*}}
Apply a change of variable $u = F_1(y)$ to get {\adb
\begin{align}
\notag &P(X_2\leq F_2^{-1}(u) | X_1 = F_1^{-1}(u)) - \lim_{u \rightarrow 0^+} P(X_2 \leq F_2^{-1}(u) |X_1 =F_1^{-1}(u)) \\
\sim&
k_{2.1} (F_1^{-1}(u))^{-2} \sim k_{2.1}(-c_1^{\inv{\eta}}u^{-\inv{\eta}})^{-2} = k^*_{2.1} u^{\frac{2}{\eta}},  
 \label{first half final}
 \end{align}}
using (\ref{asym F_1^-1}), where
{\adb
 \begin{align*}
k^{*}_{2.1} =& \left(\frac{ (\pi\eta)^{\inv{2}}\Gamma(\frac{\eta}{2})\eta}{2\Gamma(\frac{\eta+1}{2})\eta^{\frac{\eta+1}{2}}F_{t_{\eta+1}}(-\lambda_1\sqrt{\eta+1})}\right)^{\frac{2}{\eta}} \Biggl\{ -\int^{L_1}_{-\infty}f_{t_{\eta+1}}(z) \\
& \quad \times \Bigl\{ \frac{f_{t_{\eta+2}}(a(z)+b(z))b(z)\frac{\eta}{2}}{F_{t_{\eta+1}}(-\lambda_1\sqrt{\eta+1})}+ \frac{f_{t_{\eta+1}}(-\lambda_1\sqrt{\eta+1})F_{t_{\eta+2}}(a(z)+b(z))}{(F_{t_{\eta+1}}(-\lambda_1\sqrt{\eta+1}))^2}\\
& \times \lambda_1\sqrt{\eta+1}\frac{\eta}{2}\Bigr\}\,dz + \frac{1}{(\frac{1-\rho^2}{\eta+1})^{\inv{2}}}\Bigl\{\left(\frac{F_{t_{\eta+1}}(-\lambda_2\sqrt{\eta+1})}{F_{t_{\eta+1}}(-\lambda_1\sqrt{\eta+1})}\right)^{\inv{\eta}}\left(\frac{d_1-d_2(\frac{c_1}{c_2})^{\frac{2}{\eta}}}{\eta}\right)\\
& \quad + \frac{\eta}{2}(\left(\frac{F_{t_{\eta+1}}(-\lambda_2\sqrt{\eta+1})}{F_{t_{\eta+1}}(-\lambda_1\sqrt{\eta+1})}\right)^{\inv{\eta}}-\rho)\Bigr\}\times f_{t_{\eta+1}}(L_1)\frac{F_{t_{\eta+2}}(a(L_1)+b(L_1))}{F_{t_{\eta+1}}(-\lambda_1\sqrt{\eta+1})}\Biggr\}.
\end{align*}}
The rate of convergence for (\ref{original expression 2}) can therefore be obtained similarly as
\begin{equation*}
P(X_1\leq F_1^{-1}(u) |X_2 = F_2^{-1}(u)) - \lim_{u\to 0^+} P(X_1\leq F_1^{-1}(u)|X_2= F_2^{-1}(u)) \sim k_{1.2}^* u^{\frac{2}{\eta}}, 
\end{equation*}
where $k^{*}_{1.2}$ is defined analogously to $k^*_{2.1}$.

Overall, by  (\ref{origianl expression 1}) and (\ref{original expression 2})
\begin{equation}
\left| \frac{C(u,u)}{u} - \lambda_L\right| = \frac{C^{*}(u,u)}{u} = \inv{u}\int_0^u {\frac{dC^{*}(x,x)}{dx}}dx= \frac{u^{\frac{2}{\eta}}L(u)}{\eta/2+1}, , \label{defn:tail dependence
13}\end{equation}
as $u \to 0^+$, using Karamata's Theorem (See Resnick (1987), p. 17 or Seneta (1976), p.87) for regular variation at $u =0$ for the final equality,  where the slowly varying function $L(u) \sim | k^*_{2.1} + k^*_{1.2}|$ as $u\to 0^+$  is asymptotically a constant. Thus  (\ref{general approach}) obtains.
\end{proof}

Notice that if $\theta_1=\theta_2=0$ (i.e. the symmetric $t$ special case), then
\begin{align*}
& P(X_2 \leq F_2^{-1}(u) | X_1 = F_1^{-1}(u)) - \lim_{u\to 0^+} P(X_2 \leq F_2^{-1}(u) | X_1 = F_1^{-1}(u))\\
\sim& f_{t_{\eta+1}}(- \sqrt{\frac{(\eta+1)(1-\rho)}{1+\rho}})\sqrt{\frac{(\eta+1)(1-\rho)}{1+\rho}}\frac{\eta}{2}
\left(\frac{\sqrt{\pi}\Gamma(\frac{\eta}{2})}{\Gamma(\frac{\eta+1}{2})\eta^{\frac{\eta}{2}-1}}\right)^{\frac{2}{\eta}}u^{\frac{2}{\eta}}
\end{align*}
as $L_1= - \sqrt{\frac{(\eta+1)(1-\rho)}{1+\rho}}$ and $a(z)  = b(z) = \lambda_1= \lambda_2=d_1 - d_2(\frac{c_1}{c_2})^{\frac{2}{\eta}}=0$. Comparing with (\ref{first half final}), we can see that  the slowly varying bits in both are asymptotically constant, and the polynomial rate is the same. This consistency further supports the proposal that (\ref{defn:azzalini skew t}) is a proper skew extension to the symmetric multivariate $t$ distribution.\\
\section*{Acknowledgement} The elegant recursive step which simplified the proof and strengthened the statement of Theorem 2 is due to an unknown referee.

\end{document}